\documentclass[a4paper,10pt]{amsart}
\usepackage{enumerate}
\usepackage{amsthm,amsmath,amssymb,graphics,graphicx,hyperref,epstopdf,mathrsfs}
\usepackage{breqn}           
\title{Monomial ideals with large projective dimension}
\author{Guillermo Alesandroni}
\address{Department of Mathematics, Wake Forest University, 1834 Wake Forest Rd, Winston-Salem, NC 27109}
\email{alesangc@wfu.edu}

\newtheorem{theorem}{Theorem}[section]

\newtheorem{corollary}[theorem]{Corollary}
\newtheorem{lemma}[theorem]{Lemma}

\theoremstyle{definition}
\newtheorem{definition}[theorem]{Definition}

\newtheorem{example}[theorem]{Example}
\newtheorem{construction}[theorem]{Construction}
\newtheorem{notation}[theorem]{Notation}

\DeclareMathOperator{\betti}{b}
\DeclareMathOperator{\pd}{pd}
\DeclareMathOperator{\mdeg}{mdeg}
\DeclareMathOperator{\lcm}{lcm}

\DeclareMathOperator{\hdeg}{hdeg}

\DeclareMathOperator{\rank}{rank}
\DeclareMathOperator{\codim}{codim}

\begin{document}
\maketitle

\begin{abstract}
This paper has the following main results. Let $S$ be a polynomial ring in $n$ variables, over an arbitrary field. Let $\mathscr{M}$ be the family of all monomial  ideals in $S$.
\begin{enumerate}[(i)]
 \item We give an explicit characterization of all $M\in \mathscr{M}$, such that $\pd(S/M)=n$.
 \item We give the total, graded, and multigraded Betti numbers of $S/M$, in homological degree $n$, for all $M\in \mathscr{M}$.
 \item Let $M\in \mathscr{M}$. If $\pd(S/M)=n$, then $\sum\limits_{i=0}^n \betti_i(S/M)\geq 2^n$.
 \item Let $M\in \mathscr{M}$. If $M$ is Artinian and $\betti_n(S/M)=1$, then $M$ is a complete intersection. 
 \end{enumerate}
 \end{abstract}

 \section{Introduction}
 Let $S=k[x_1,\ldots,x_n]$ be a polynomial ring in $n$ variables, over a field $k$. The title of this paper makes reference to those monomial ideals $M$ in $S$, for which the quotient module $S/M$ has projective dimension $n$, and the present work is entirely concerned with the study of such ideals.
 
 We begin to examine projective dimension $n$ in the context of squarefree monomial ideals. We show that the only squarefree monomial ideal $M$ for which the projective dimension of $S/M$ equals $n$ is the maximal ideal $M=(x_1,\ldots,x_n)$. This result turns out to be instrumental in the proof of a later theorem, where we characterize the class of all monomial ideals with large projective dimension. This characterization, in turn, is an avenue to three results that we discuss below.
 
 General consensus says that the problem of describing the Betti numbers of an arbitrary monomial ideal of $S$ is utopian. In homological degree $n$, however, such description is particularly simple. In fact, we give the total, graded, and multigraded Betti numbers of $S/M$, in homological degree $n$, for every monomial ideal $M$ of $S$.
 
 Another theorem proven in this article states that when the quotient $S/M$ has projective dimension $n$, the sum of its Betti numbers is at least $2^n$. This result, already known for Artinian monomial ideals [Ch, CE], is related to the Buchsbaum-Eisenbud, Horrocks conjecture, which has been investigated and generalized over the course of the years [CE], [PS, Conjectures 6.5, 6.6, and 6.7]. The proof of our theorem has strong combinatorial flavor.
 
 Finally, we show that when $M$ is Artinian and the $n^{th}$ Betti number of $S/M$ is 1, $M$ must be of the form $M=(x_1^{\alpha_1},\ldots,x_n^{\alpha_n})$, where the $\alpha_i$ are
 positive integers. Combining this result with [Pe, Theorem 25.7] (a criterion for $S/M$ to be Gorenstein), we obtain the following. If $\betti_n(S/M)=1$, then $S/M$ is Cohen-Macaulay if and only if $S/M$ is Gorenstein if and only if $M=(x_1^{\alpha_1},\ldots,x_n^{\alpha_n})$, for some ${\alpha_1},\ldots,{\alpha_n}\geq 1$.
 
 The organization of the article is as follows. Section 2 is about background and notation. Sections 3 and 4 prepare the ground to characterize all monomial ideals with large projective dimension. 
 This characterization is the content of section 5. Section 6 is the heart of this work; it is in this section that we prove the three theorems advertised above.
 
 \section{Background and Notation}
 Throughout this paper $k$ is an arbitrary field, and $S$ represents a polynomial ring over $k$, in a finite number variables. The letter $n$ is always used to denote the number of variables of $S$. The letter $M$ represents a monomial ideal
in $S$. With minor modifications, the construction that we give below can be found in [Me].
 
\begin{construction}
Let $M$ be generated by a set of monomials $\{l_1,\ldots,l_q\}$. For every subset $\{l_{i_1},\ldots,l_{i_s}\}$ of $\{l_1,\ldots,l_q\}$, with $1\leq i_1<\ldots<i_s\leq q$, 
we create a formal symbol $[l_{i_1},\ldots,l_{i_s}]$, called a \textbf{Taylor symbol}. The Taylor symbol associated to $\varnothing$ is denoted by $[\varnothing]$.
For each $s=0,\ldots,q$, set $F_s$ equal to the free $S$-module with basis $\{[l_{i_1},\ldots,l_{i_s}]:1\leq i_1<\ldots<i_s\leq q\}$ given by the 
${q\choose s}$ Taylor symbols corresponding to subsets of size $s$. That is, $F_s=\bigoplus\limits_{i_1<\ldots<i_s}S[l_{i_1},\ldots,l_{i_s}]$ 
(note that $F_0=S[\varnothing]$). Define
\[f_0:F_0\rightarrow S/M\]
\[s[\varnothing]\mapsto f_0(s[\varnothing])=s.\]
For $s=1,\ldots,q$, let $f_s:F_s\rightarrow F_{s-1}$ be given by
\[f_s\left([l_{i_1},\ldots,l_{i_s}]\right)=
 \sum\limits_{j=1}^s\dfrac{(-1)^{j+1}\lcm(l_{i_1},\ldots,l_{i_s})}{\lcm(l_{i_1},\ldots,\widehat{l_{i_j}},\ldots,l_{i_s})}
 [l_{i_1},\ldots,\widehat{l_{i_j}},\ldots,l_{i_s}]\]
 and extended by linearity.
 The \textbf{Taylor resolution} $\mathbb{T}_{l_1,\ldots,l_q}$ of $S/M$ is the exact sequence
 \[\mathbb{T}_{l_1,\ldots,l_q}:0\rightarrow F_q\xrightarrow{f_q}F_{q-1}\rightarrow\cdots\rightarrow F_1\xrightarrow{f_1}F_0\xrightarrow{f_0} 
 S/M\rightarrow0.\]
 \end{construction}
We define the \textbf{multidegree} of a Taylor symbol $[l_{i_1},\ldots,l_{i_s}]$, denoted $\mdeg[l_{i_1},\ldots,l_{i_s}]$, as follows:
  $\mdeg[l_{i_1},\ldots,l_{i_s}]=\lcm(l_{i_1},\ldots,l_{i_s})$. 
 The \textbf{degree} of a Taylor symbol $[l_{i_1},\ldots,l_{i_s}]$, denoted $\deg[l_{i_1},\ldots,l_{i_s}]$, is the total degree of the multidegree $\mdeg[l_{i_1},\ldots,l_{i_s}]$. For example, if 
 $\mdeg[l_{i_1},\ldots,l_{i_s}]=x^2y^3$, then $\deg[l_{i_1},\ldots,l_{i_s}]=5$.

  \textit{Note}:
  In our construction above, the generating set $\{l_1,\ldots,l_q\}$ is not required to be minimal. Thus, $S/M$ has many Taylor resolutions. We reserve the notation 
  $\mathbb{T}_M$ for the Taylor resolution of $S/M$, determined by the minimal generating set of $M$. (Although some authors define a single Taylor resolution of $S/M$, our construction is general, like in [Ei].)

\begin{definition}

 Let $M$ be a monomial ideal, and let
 \[\mathbb{F}:\cdots\rightarrow F_i\xrightarrow{f_i}F_{i-1}\rightarrow\cdots\rightarrow F_1\xrightarrow{f_1}F_0\xrightarrow{f_0} S/M\rightarrow 0\]
be a free resolution of $S/M$. 
We say that a basis element $[\sigma]$ of $\mathbb{F}$ has \textbf{homological degree i}, denoted $\hdeg[\sigma]=i$, if 
$[\sigma] \in F_i$. $\mathbb{F}$ is said to be a \textbf{minimal resolution} if for every $i$, the differential matrix $\left(f_i\right)$ of $\mathbb{F}$
has no invertible entries.
\end{definition}
\begin{definition}
Let $M$ be a monomial ideal, and let
 \[\mathbb{F}:\cdots\rightarrow F_i\xrightarrow{f_i}F_{i-1}\rightarrow\cdots\rightarrow F_1\xrightarrow{f_1}F_0\xrightarrow{f_0} S/M\rightarrow 0\]
be a minimal free resolution of $S/M$.
\begin{itemize}
 \item For every $i\geq 0$, the $i^{th}$ \textbf{Betti number} $\betti_i\left(S/M\right)$ of $S/M$ is $\betti_i\left(S/M\right)=\rank(F_i)$.
 \item For every $i,j\geq 0$, the \textbf{graded Betti number} $\betti_{i,j}\left(S/M\right)$ of $S/M$, in homological degree $i$ and internal degree $j$,
is \[\betti_{i,j}\left(S/M\right)=\#\{\text{basis elements }[\sigma]\text{ of }F_i:\deg[\sigma]=j\}.\]
\item For every $i\geq 0$, and every monomial $l$, the \textbf{multigraded Betti number} $\betti_{i,l}\left(S/M\right)$ of $S/M$, in homological degree $i$ and multidegree $l$,
is \[\betti_{i,l}\left(S/M\right)=\#\{\text{basis elements }[\sigma]\text{ of }F_i:\mdeg[\sigma]=l\}.\]
\item The \textbf{projective dimension} $\pd\left(S/M\right)$ of $S/M$ is \[\pd\left(S/M\right)=\max\{i:\betti_i\left(S/M\right)\neq 0\}.\]
\end{itemize}
\end{definition}

\begin{definition}
Let $L$ be a set of monomials, and let $M$ be a monomial ideal with minimal generating $G$.
\begin{itemize}
\item An element $m\in L$ is a \textbf{dominant monomial} (in $L$) if there is a variable $x$, such that for all $m'\in L\setminus\{m\}$, the exponent with 
which $x$ appears in the factorization of $m$ is larger than the exponent with which $x$ appears in the factorization of $m'$. In this case, we say that $m$ is dominant in $x$, and $x$ is a \textbf{dominant variable} for $m$.
\item $L$ is called a \textbf{dominant set} if each of its monomials is dominant. 
\item $M$ is called a \textbf{dominant ideal} if $G$ is a dominant set. 
\item If $G'$ is a dominant set contained in $G$, we will say that $G'$ is a \textbf{dominant subset} of $G$.  (This does not mean that the elements of $G'$ are dominant in $G$, as the concept of dominant monomial always depends on a reference set.)

\end{itemize}
\end{definition}

 \begin{example}\label{example 1}
  Let $M$ be minimally generated by $G=\{a^2b,ab^3c,bc^2,a^2c^2\}$, and let $G'=\{a^2b,ab^3c,bc^2\}$. Note that $ab^3c$ is the only dominant monomial in $G$, being $b$ a dominant variable for $ab^3c$. It is easy to check that
  $G'$ is a dominant set and, given that $G'\subseteq G$, $G'$ is a dominant subset of $G$. (Incidentally, notice that two of the dominant monomials in $G'$ are not dominant in $G$.) Finally, the ideal $M'$, minimally generated by $G'$, is a dominant ideal, for $G'$ is a dominant set. 
  \end{example}

 For a more detailed treatment of the concept of dominance, see [Al].
 
 \section{Auxiliary Results}

 \textit{Note:} since the free modules of $\mathbb{T}_M$ are graded by multidegree, if $a=\alpha x_1^{\alpha_1}\ldots x_n^{\alpha_n}$, with $\alpha \in k \setminus \{0\}$, then 
 \[\mdeg(a[\sigma])=x_1^{\alpha_1}\ldots x_n^{\alpha_n} \mdeg[\sigma].\]
 
 \begin{theorem}\label{theorem 1.2}
 Let $M$ be a squarefree monomial ideal in $S=k[x_1,\ldots,x_n]$. If $\pd(S/M)=n$, then $M=(x_1,\ldots,x_n)$. 
 \end{theorem}
 
 \begin{proof}
 Let $G$ be the minimal generating set of $M$. Let 
 \[\mathbb{F}: 0\rightarrow F_n\xrightarrow{f_n} F_{n-1}\cdots F_1\xrightarrow{f_1} F_0 \xrightarrow{f_0} S/M \rightarrow 0\]
 be a minimal resolution of $S/M$, obtained from $\mathbb{T}_M$ by means of consecutive cancellations. Let $[\theta]$ be a basis element of $F_n$ and let $f_n[\theta]=\sum a_i [\tau_i]$. By the minimality of  $\mathbb{F}$, none of the $a_i$ is invertible, and at least one of the $a_i$ is not zero, say $a_r\neq 0$. It follows that 
 $\mdeg[\theta]=\mdeg(a_r[\tau_r]) $. Let $[\sigma_n]=[\theta]$ and $[\sigma_{n-1}]=[\tau_r]$. Note that $\deg[\sigma_{n-1}]<\deg[\sigma_n]$, and $\mdeg[\sigma_{n-1}]\mid \mdeg[\sigma_n]$.\\
 Suppose $[\sigma_n],\ldots,[\sigma_{n-j}]$ are basis elements of $F_n,\ldots, F_{n-j}$, respectively, such that, for all $i=1,\ldots,j$, $\deg[\sigma_{n-i}]<\deg[\sigma_{n-i+1}]$, and $\mdeg[\sigma_{n-i}]\mid\mdeg[\sigma_{n-i+1}]$.\\ 
 Let $f_{n-j}[\sigma_{n-j}]=\sum \betti_i [\xi_i]$. By the minimality of $\mathbb{F}$, none of the $\betti_i$ is invertible, and at least one of the $\betti_i$ is not zero, say $\betti_s\neq 0$. It follows, $\mdeg[\sigma_{n-j}]=\mdeg(\betti_s[\xi_s])$.\\
 Let $[\sigma_{n-j-1}]=[\xi_s]$. Note that $\deg[\sigma_{n-j-1}]<\deg[\sigma_{n-j}]$, and $\mdeg[\sigma_{n-j-1}]\mid \mdeg[\sigma_{n-j}]$. Thus, we can recursively define a sequence $[\sigma_n],\ldots,[\sigma_1]$ of basis elements of $F_n,\ldots,F_1$, respectively, such that $1\leq \deg[\sigma_1]<\ldots<\deg[\sigma_n] \leq n$, and $\mdeg[\sigma_{n-i}]\mid\mdeg[\sigma_{n-i+1}]$, for all $i=1,\ldots,n-1$.
 Thus, we must have that $\deg[\sigma_i]=i$, for all $i=1,\ldots,n$. \\
 Since $\deg[\sigma_1]=1$, there is some variable, say $x_1$, such that $\mdeg[\sigma_1]=x_1$. By definition, $\mdeg[\sigma_1]$ is the least common multiple of a minimal generator. Thus, $x_1$ must be in the minimal generating set $G$ of $M$, and we must have $[\sigma_1]=[x_1]$. Since $\deg[\sigma_2]=2$, and $x_1=\mdeg[\sigma_1]\mid \mdeg[\sigma_2]$, there is some variable, say $x_2$, such that $\mdeg[\sigma_2]=x_1x_2$. By definition, $\mdeg[\sigma_2]$ is the least common multiple of two minimal generators, one of which must be divisible by $x_1$. But the only minimal generator that is divisible by $x_1$ is $x_1$ itself. Thus, $[\sigma_2]$ must be of the form $[\sigma_2]=[x_1,l]$. Since $l$ is not divisible by any of $x_3,\ldots,x_n$, the only possibility is $l=x_2$. Thus, $x_2$ must be in $G$, and we must have $[\sigma_2]=[x_1,x_2]$. Suppose that, for all $i=1,\ldots,j$, $x_i$ is in $G$, and $[\sigma_i]=[x_1,\ldots,x_i]$. Since $\deg[\sigma_{j+1}]=j+1$, and $x_1\ldots x_j=\mdeg[\sigma_j] \mid \mdeg[\sigma_{j+1}]$, there is some variable, say $x_{j+1}$, such that $\mdeg[\sigma_{j+1}]=x_1\ldots x_{j+1}$. By definition, $\mdeg[\sigma_{j+1}]$ is the least common multiple of $j+1$ minimal generators. Given that, for all $i=1,\ldots,j$, $x_i$ is the only minimal generator divisible by $x_i$, $[\sigma_{j+1}]$ must be of the form $[\sigma_{j+1}]=[x_1,\ldots,x_j,l]$. Therefore, $l$ must be divisible by $x_{j+1}$, and $l$ must not be divisible by any of $x_{j+2},\ldots,x_n$. Then, the only possibility is $l=x_{j+1}$. It follows that $x_{j+1}$ is in $G$, and $[\sigma_{j+1}]=[x_1,\ldots,x_{j+1}]$. We have proven that $x_1,\ldots,x_n \in G$, and $[\sigma_1]=[x_1],\ldots,[\sigma_n]=[x_1,\ldots,x_n]$. Hence, $M=(x_1,\ldots,x_n)$.   
 \end{proof}
 
 The notation that we introduce below retains its meaning until the end of this section.
 
 \begin{notation}
 Let $\mathbb{F}_0$ be a free resolution of $S/M$, obtained from $\mathbb{T}_M$ by means of consecutive cancellations. If $c_{\gamma_0 \delta_0}^{(0)}$ is an invertible entry of $\mathbb{F}_0$, determined by basis elements 
 $[\delta_0],[\gamma_0]$ of multidegree $l_0$ (cf. [Al, Remark 3.4]), let $\mathbb{F}_1$ be the resolution of $S/M$, such that
 \[\mathbb{F}_0=\mathbb{F}_1 \oplus \left(0 \rightarrow S[\delta_0] \rightarrow S[\gamma_0] \rightarrow 0 \right).\]
 Assume that $\mathbb{F}_{k-1}$ has been defined. If $c_{\gamma_{k-1} \delta_{k-1}}^{(k-1)}$ is an invertible entry of $\mathbb{F}_{k-1}$, determined by basis elements $[\delta_{k-1}]$, $[\gamma_{k-1}]$ of multidegree $l_{k-1}$, let $\mathbb{F}_k$ be the resolution of $S/M$, such that 
 \[\mathbb{F}_{k-1}=\mathbb{F}_k \oplus \left(0 \rightarrow S[\delta_{k-1}] \rightarrow S[\gamma_{k-1}] \rightarrow 0 \right).\]
 \end{notation}
 
 \begin{theorem}\label{1}
 Suppose that $\mathbb{F}_0,\ldots,\mathbb{F}_v$ are resolutions of $S/M$, defined as above. If $c_{\pi \theta}^{(0)}$ is a noninvertible entry of $\mathbb{F}_0$, determined by basis elements $[\theta]$, $[\pi]$ of multidegree $l$, and for all $0\leq i\leq v$, $l \neq l_i$, then the entry $c_{\pi \theta}^{(v)}$ of $\mathbb{F}_v$, determined by $[\theta]$, $[\pi]$ is noninvertible.
 \end{theorem}
 
 \begin{proof}
 By induction on $v$. If $v=0$, there is nothing to prove.\\
 Let us assume that the statement holds for $v-1$. If $\hdeg[\theta]\neq \hdeg[\delta_{v-1}]$, then $c_{\pi \theta}^{(v)}= c_{\pi \theta}^{(v-1)}$, by [Al, Lemma 3.2(iv)], and the result holds by induction hypothesis.\\
 If $\hdeg[\theta]=\hdeg[\delta_{v-1}]$, then
 \[c_{\pi \theta}^{(v)}=c_{\pi \theta}^{(v-1)}- \dfrac{c_{\pi \delta_{v-1}}^{(v-1)} c_{\gamma_{v-1} \theta}^{(v-1)}}{c_{\gamma_{v-1} \delta_{v-1}}^{(v-1)}},\]
 by [Al, Lemma 3.2(iii)]. By induction hypothesis, $c_{\pi \theta}^{(v-1)}$ is noninvertible. Also, since $\mdeg[\pi]=l \neq l_{v-1}=\mdeg[\delta_{v-1}]$, the entry $c_{\pi \delta_{v-1}}^{(v-1)}$ is noninvertible. Hence, the product 
 $c_{\pi \delta_{v-1}}^{(v-1)} c_{\gamma_{v-1} \theta}^{(v-1)}$ is noninvertible. Since $c_{\gamma_{v-1} \delta_{v-1}}^{(v-1)}$ is invertible, the quotient $ \dfrac{c_{\pi \delta_{v-1}}^{(v-1)} c_{\gamma_{v-1} \theta}^{(v-1)}}{c_{\gamma_{v-1} \delta_{v-1}}^{(v-1)}}$ is noninvertible. Finally, $c_{\pi \theta}^{(v)}$ must be noninvertible, for the difference of two noninvertible monomials is noninvertible.
 \end{proof}
 
 \begin{corollary}\label{corollary 1.5}
 Let $\mathbb{F}$ be a free resolution of $S/M$, obtained from $\mathbb{T}_M$ by means of consecutive cancellations. Suppose that $l$ is a multidegree, such that every entry $c_{\pi \theta}$ of $\mathbb{F}$, determined by basis elements $[\pi]$, $[\theta]$ of multidegree $l$, is noninvertible. Then, for all $i \geq 0$,
 \[\betti_{i,l}(S/M)= \# \{[\sigma] \in \mathbb{F}: \mdeg[\sigma]=l, \hdeg[\sigma]=i\}.\] 
 \end{corollary}
 
 \begin{proof}
 Since the process of making consecutive cancellations must eventually terminate, there is a number $v\geq 0$, such that $\mathbb{F}_0=\mathbb{F},\mathbb{F}_1,\ldots,\mathbb{F}_v$ are resolutions of $S/M$ defined as above, and $\mathbb{F}_v$ is a minimal resolution. That is:
 
\[ \begin{array}{rcl}
 \mathbb{F}  =\mathbb{F}_0 & = &\mathbb{F}_1 \oplus \left(0\rightarrow S[\delta_0]\rightarrow S[\gamma_0]\rightarrow 0 \right) \\
  \mathbb{F}_1 & =  &\mathbb{F}_2 \oplus \left(0\rightarrow S[\delta_1]\rightarrow S[\gamma_1]\rightarrow 0 \right) \\
                           & \vdots \\
  \mathbb{F}_{v-1} &=&\mathbb{F}_v     \oplus \left(0\rightarrow S[\delta_{v-1}]\rightarrow S[\gamma_{v-1}]\rightarrow 0 \right),                     
 \end{array}\] 
 where $\mathbb{F}_v$ is a minimal resolution of $S/M$. Suppose, by means of contradiction, that at least one of $[\delta_0],\ldots,[\delta_{v-1}]$ has multidegree $l$. Let $i$ be the smallest integer such that $\mdeg[\delta_i]=l$.
 Since $[\delta_i]$, $[\gamma_i]$ have multidegree $l$, the entry $c_{\gamma_i \delta_i}^{(0)}$ of $\mathbb{F}_0=\mathbb{F}$ is noninvertible by hypothesis. Since the multidegrees of $[\delta_0],\ldots,[\delta _{i-1}]$ are not equal to $l$, it follows from Theorem \ref{1} that the entry $c_{\gamma_i \delta_i}^{(i)}$ of $\mathbb{F}_i$, determined by $[\delta_i]$, $[\gamma_i]$, is noninvertible. However, the fact that $\mathbb{F}_{i+1}$ is obtained from $\mathbb{F}_i$ by doing the cancellation $0\rightarrow S[\delta_i]\rightarrow S[\gamma_i]\rightarrow 0$, implies that $c_{\gamma_i \delta_i}^{(i)}$ is invertible, a contradiction.\\
 We have proven that the multidegrees of $[\delta_0],\ldots,[\delta_{v-1}]$ are not equal to $l$, and hence, the multidegrees of $[\gamma_0],\ldots,[\gamma_{v-1}]$ are not equal to $l$. Thus, the basis of the minimal resolution $\mathbb{F}_v$ is obtained from the basis of $\mathbb{F}_0=\mathbb{F}$ by removing basis elements of multidegree not equal to $l$. Then the basis elements of $\mathbb{F}_0=\mathbb{F}$, with multidegree $l$, are the same as the basis elements of $\mathbb{F}_v$, with multidegree $l$.
 \end{proof}
 
 \section{Isomorphism Theorems}
 
 \begin{construction}\label{const 1}
 Let $M=(m_1,\ldots,m_q)$, where $m_i=x_1^{\alpha_{i1}}\ldots x_n^{\alpha_{in}}$, for all $i=1,\ldots,q$. Let $m=\lcm(m_1,\ldots,m_q)$. Then $m$ factors as $m=x_1^{\alpha_1} \ldots x_n^{\alpha_n}$, where 
 $\alpha_j=\max(\alpha_{1j}, \ldots, \alpha_{qj} )$, for all $j=1,\ldots,n$. For all $i=1,\ldots,q$, we define the monomial $m'_i$ as follows:
 \[m'_i=x_1^{\beta_{i1}} \ldots x_n^{\beta_{in}} \text{, where } \beta_{ij}= \begin{cases} 
 														\alpha_j \text{, if } \alpha_{ij}=\alpha_j \\
														0 \text{, otherwise.}
														\end{cases}\]
Let $M'=(m'_1,\ldots,m'_q)$. The ideal $M'$ will be referred to as the \textbf{twin ideal} of $M$.
 \end{construction}
 The notation introduced in Construction \ref{const 1} retains its meaning until the end of this section.
 
 \begin{example}
 Let $M=(m_1=a^3b^2,m_2=a^3c,m_3=ac^2,m_4=bc^2)$. Then $m=a^3b^2c^2$, and $M'=(m'_1=a^3b^2,m'_2=a^3,m'_3=c^2,m'_4=c^2)$.
 \end{example}
 Note that $M'$ is not minimally generated by $\{m'_1,m'_2,m'_3,m'_4\}$ because some generators of this set are redundant; moreover, some monomials are duplicated. However, nonminimal generating sets will play an important role in this section.
 
 \begin{theorem} \label{theorem AiBi}
 For every $i\geq 0$, there is a bijective correspondence between the basis elements of $\mathbb{T}_{m'_1,\ldots, m'_q}$ with multidegree $m$ and homological degree $i$, and the basis elements of $\mathbb{T}_M$ with multidegree $m$ and homological degree $i$.
 \end{theorem} 
 
 \begin{proof}
 Let $A_i=\{[\sigma]\in \mathbb{T}_{m'_1,\ldots,m'_q}: \hdeg[\sigma]=i \text{ and } \mdeg[\sigma]=m\}$. Let $B_i=\{[\sigma]\in \mathbb{T}_M:\hdeg[\sigma]=i \text{ and } \mdeg[\sigma]=m\}$.\\
 Let $f_i: A_i\rightarrow B_i$ be defined by $f_i[m'_{r_1},\ldots,m'_{r_i}]=[m_{r_1},\ldots,m_{r_i}]$.\\
 \begin{itemize}
 \item $f_i$ is well defined.\\
 If $\mdeg[m'_{r_1},\ldots,m'_{r_i}]=m$, then $m \mid \mdeg[m_{r_1},\ldots,m_{r_i}]$, for each $m_{r_k}$ is a multiple of $m'_{r_k}$. On the other hand, $\mdeg[m_{r_1},\ldots,m_{r_i}]\mid \lcm(m_1,\ldots,m_q)=m$. Thus, $\mdeg[m_{r_1},\ldots,m_{r_i}]=m$.
 \item $f_i$ is one-to-one.\\
 If If $[m'_{r_1},\ldots,m'_{r_i}]$, $[m'_{s_1},\ldots,m'_{s_i}]$ are different basis elements of $A_i$, then $r_k \neq s_k$, for some $k$. Therefore, $[m_{r_1},\ldots,m_{r_i}] \neq [m_{s_1},\ldots,m_{s_i}]$. 
 \item $f_i$ is onto.\\
 Let $[m_{r_1},\ldots,m_{r_i}]\in B_i$. Then $\mdeg[m_{r_1},\ldots,m_{r_i}]=m=x_1^{\alpha_1}\ldots x_n^{\alpha_n}$. Hence, there is some 
 $m_r \in \{m_{r_1},\ldots,m_{r_i}\}$, such that $m_r=x_1^{\alpha_{r1}} \ldots x_n^{\alpha_{rn}}$, with $\alpha_{r_1}=\alpha_1$. Thus, $m'_r=x_1^{\beta_{r1}}\ldots x_n^{\beta_{rn}}$, with $\beta_{r1}=\alpha_1$. This means that $x_1^{\alpha_1} \mid m'_r$, and therefore, $x_1^{\alpha_1}\mid \mdeg[m'_{r_1},\ldots,m'_{r_i}]$. Similarly, $x_2^{\alpha_2},\ldots,x_n^{\alpha_n}\mid \mdeg[m'_{r_1},\ldots,m'_{r_i}]$. It follows that $m\mid \mdeg[m'_{r_1},\ldots,m'_{r_i}]$. Since $\mdeg[m'_{r_1},\ldots,m'_{r_i}]\mid m$, we conclude that $\mdeg[m'_{r_1},\ldots,m'_{r_i}]=m$, and $[m'_{r_1},\ldots,m'_{r_i}] \in A_i$. 
 \end{itemize}  
 \end{proof}
 
 \begin{notation}
 Let $A= \bigcup A_i$, and $B= \bigcup B_i$, where $A_i$, $B_i$ are the sets defined in Theorem \ref{theorem AiBi}. Let $f:A\rightarrow B$ be given by $f([\sigma])=f_i[\sigma]$, if $[\sigma] \in A_i$, where $f_i$ is the map defined in Theorem \ref{theorem AiBi}. Note that $A$ and $B$ are the sets of all basis elements with multidegree $m$ in $\mathbb{T}_{m'_1,\ldots,m'_q}$ and $\mathbb{T}_M$, respectively, and $f:A\rightarrow B$ is a bijection that sends elements with multidegree $m$ and homological degree $i$ to elements with multidegree $m$ and homological degree $i$.
 \end{notation}
 
 \begin{theorem}\label{4}
 If $a_{\pi \theta}$ is an entry of $\mathbb{T}_{m'_1,\ldots,m'_q}$, determined by elements $[\theta],[\pi] \in A$, then $f[\theta]$, $f[\pi]$ determine an entry $b_{\pi \theta}$ of $\mathbb{T}_M$, such that 
 $b_{\pi \theta}=a_{\pi \theta}$.
 \end{theorem}
 
 \begin{proof}
 Since $[\theta],[\pi]$ appear in consecutive homological degrees, so do $f[\theta],f[\pi]$. Thus, $f[\theta],f[\pi]$ determine an entry $b_{\pi \theta}$ of $\mathbb{T}_M$. If $[\pi]$ is a facet of $[\theta]$, then $f[\pi]$ is a facet of $f[\theta]$, and these elements are of the form 
 \begin{align*}
 [\theta] &=[m'_{r_1},\ldots,m'_{r_i}]; & [\pi]=[m'_{r_1},\ldots,\widehat{m'_{r_t}},\ldots,m'_{r_i}]\\
 f[\theta] &=[m_{r_1},\ldots,m_{r_i}];  &f[\pi]=[m_{r_1},\ldots, \widehat{m_{r_t}},\ldots,m_{r_i}]
 \end{align*}
 Thus,
 \[b_{\pi \theta}=(-1)^{t+1}\dfrac{\mdeg f[\theta]}{\mdeg f[\pi]}=(-1)^{t+1}\dfrac{m}{m}=(-1)^{t+1}\dfrac{\mdeg [\theta]}{\mdeg[\pi]}=a_{\pi \theta}.\]
 On the other hand, if $[\pi]$ is not a facet of $[\theta]$, $f[\pi]$ cannot be a facet of $f[\theta]$. Thus, $b_{\pi \theta}=0=a_{\pi \theta}$.
 \end{proof}
 
 \begin{notation}
 Let $\mathbb{F}_0=\mathbb{T}_{m'_1,\ldots,m'_q}$. If there is an invertible entry $a_{\pi_0 \theta_0}^{(0)}$ of $\mathbb{F}_0$, determined by elements $[\theta_0], [\pi_0] \in A$, let $\mathbb{F}_1$ be the resolution of $S/M'$ such that
 \[\mathbb{F}_0=\mathbb{F}_1 \oplus \left(0\rightarrow S[\theta_0] \rightarrow S[\pi_0] \rightarrow 0 \right). \]
 Let us assume that $\mathbb{F}_{k-1}$ has been defined. \\
 If there is an invertible entry  $a_{\pi_{k-1}\theta_{k-1}}^{(k-1)}$ of $\mathbb{F}_{k-1}$, determined by elements $[\theta_{k-1}],[\pi_{k-1}]$ of $A$, let $\mathbb{F}_k$ be the resolution of $S/M'$ such that
 \[\mathbb{F}_{k-1}=\mathbb{F}_k \oplus \left(0 \rightarrow S[\theta_{k-1}] \rightarrow S[\pi_{k-1}] \rightarrow 0 \right).\]
 \end{notation}
 
 \begin{theorem}\label{theorem exist}
 Suppose that $\mathbb{F}_0,\ldots\mathbb{F}_u$ are resolutions of $S/M'$, defined as above. Then
 \begin{enumerate}[(i)]
 \item There exist resolutions $\mathbb{G}_0\ldots,\mathbb{G}_u$ of $S/M$, defined as follows
 \[\mathbb{G}_0=\mathbb{T}_M; \mathbb{G}_{k-1}=\mathbb{G}_k \oplus \left(0\rightarrow Sf[\theta_{k-1}]\rightarrow Sf[\pi_{k-1}] \rightarrow 0\right).\]
 \item If $a_{\tau \sigma}^{(u)}$ is an entry of $\mathbb{F}_u$, determined by elements $[\sigma],[\tau] \in A$, then $f[\sigma],f[\tau]$ are in the basis of $\mathbb{G}_u$ and determine an entry $b_{\tau\sigma}^{(u)}$ of 
 $\mathbb{G}_u$, such that $b_{\tau\sigma}^{(u)}=a_{\tau\sigma}^{(u)}$.
 \end{enumerate}
 \end{theorem}
 
 \begin{proof}
 The proof is by induction on $u$. If $u=0$, (i) and (ii) are the content of Theorem \ref{4}.\\
 Let us assume that parts (i) and (ii) hold for $u-1$.
 We will prove parts (i) and (ii) for $u$.\\
 (i) We need to show that $\mathbb{G}_u$ can be defined by the rule
 \[\mathbb{G}_{u-1}=\mathbb{G}_u\oplus\left(0\rightarrow Sf[\theta_{u-1}]\rightarrow Sf[\pi_{u-1}]\rightarrow 0\right).\]
 In other words, we must show that $f[\theta_{u-1}]$, $f[\pi_{u-1}]$ are in the basis of $\mathbb{G}_{u-1}$, and the entry $b^{(u-1)}_{\pi_{u-1}\theta_{u-1}}$ of
 $\mathbb{G}_{u-1}$, determined by them, is invertible. But this follows from induction hypothesis and the fact that $a^{(u-1)}_{\pi_{u-1}\theta_{u-1}}$ is
 invertible.\\
 (ii) Notice that the basis of $\mathbb{F}_u$ is obtained from the basis of $\mathbb{F}_{u-1}$, by eliminating $[\theta_{u-1}],[\pi_{u-1}]$. This means
 that $[\sigma],[\tau]$ are in the basis of $\mathbb{F}_{u-1}$, and the pairs $\left([\sigma],[\tau]\right)$, $\left([\theta_{u-1}],[\pi_{u-1}]\right)$ are
 disjoint. Then by induction hypothesis, $f[\sigma]$, $f[\tau]$ are in the basis of $\mathbb{G}_{u-1}$, and because $f$ is a bijection, 
 $\left(f[\sigma],f[\tau]\right)$, $\left(f[\theta_{u-1}],f[\pi_{u-1}]\right)$ are disjoint pairs. Since the basis of $\mathbb{G}_u$ is obtained from the
 basis of $\mathbb{G}_{u-1}$, by eliminating $f[\theta_{u-1}],f[\pi_{u-1}]$, we must have that $f[\sigma],f[\tau]$ are in the basis of $\mathbb{G}_u$.
 Finally, we need to prove that $b^{(u)}_{\tau\sigma}=a^{(u)}_{\tau\sigma}$. By [Al, Lemma 3.2(iv)], if 
 $\hdeg[\sigma]\neq\hdeg[\theta_{u-1}]$, then $a^{(u)}_{\tau\sigma}=a^{(u-1)}_{\tau\sigma}$. In this case, we must also have that 
 $\hdeg f[\sigma]\neq \hdeg f[\theta_{u-1}]$, which implies that $b^{(u)}_{\tau\sigma}=b^{(u-1)}_{\tau\sigma}$, by the same lemma. Then, by induction hypothesis, 
 $b^{(u)}_{\tau\sigma}=b^{(u-1)}_{\tau\sigma}=a^{(u-1)}_{\tau\sigma}=a^{(u)}_{\tau\sigma}$. On the other hand, if 
 $\hdeg[\sigma]=\hdeg[\theta_{u-1}]$, then $\hdeg f[\sigma]=\hdeg f[\theta_{u-1}]$. Combining the induction hypothesis with [Al, Lemma 3.2(iii)], we 
 obtain
 \[b^{(u)}_{\tau\sigma}=b^{(u-1)}_{\tau\sigma}-\dfrac{b^{(u-1)}_{\tau\theta_{u-1}}b^{(u-1)}_{\pi_{u-1}\sigma}}{b^{(u-1)}_{\pi_{u-1}\theta_{u-1}}}=
 a^{(u-1)}_{\tau\sigma}-\dfrac{a^{(u-1)}_{\tau\theta_{u-1}}a^{(u-1)}_{\pi_{u-1}\sigma}}{a^{(u-1)}_{\pi_{u-1}\theta_{u-1}}}=
 a^{(u)}_{\tau\sigma}\]
 \end{proof}
 
 \textit{Note}: Since the process of making consecutive cancellations between pairs of basis elements of $A$ must eventually terminate, there is an integer $u \geq 0$, such that $\mathbb{F}_0,\ldots,\mathbb{F}_u$ are resolutions of $S/M'$ defined as above, and each entry $a_{\pi\theta}^{(u)}$ of $\mathbb{F}_u$, determined by elements $[\theta],[\pi]$ of $A$, is noninvertible. For the rest of this section $u$ is such an integer and $\mathbb{F}_0,\ldots,\mathbb{F}_u$ are such resolutions. Moreover, the resolutions $\mathbb{G}_0,\ldots,\mathbb{G}_u$, constructed in Theorem \ref{theorem exist} are also fixed until the end of this section.
 
 \begin{theorem}\label{theorem 6}
 If $b_{\pi\theta}^{(u)}$ is an entry of $\mathbb{G}_u$, determined by basis elements $f[\theta],f[\pi]$ of $B$, then $b_{\pi\theta}^{(u)}$ is noninvertible.
 \end{theorem}
 
 \begin{proof}
 This is an immediate consequence of Theorem \ref{theorem exist}(ii), and the fact that $a_{\pi\theta}^{(u)}$ is noninvertible.
 \end{proof}
 
 \begin{theorem}\label{theorem 7}
 For every $i \geq 0$, there is a bijective correspondence between the basis elements with multidegree $m$ and homological degree $i$ of $\mathbb{F}_u$ and $\mathbb{G}_u$.
 \end{theorem}
 
 \begin{proof}
 The set of basis elements of $\mathbb{F}_u$ with multidegree $m$ and homological degree $i$ is given by
 \[A'_i=A_i \setminus \{[\theta_0],[\pi_0],\ldots,[\theta_{u-1}],[\pi_{u-1}]\}.\]
 Likewise, the set of basis elements of $\mathbb{G}_u$ with multidegree $m$ and homological degree $i$ is given by
 \[B'_i=B_i \setminus \{f[\theta_0],f[\pi_0],\ldots,f[\theta_{u-1}],f[\pi_{u-1}]\}.\]
 Since $A_i \xrightarrow{f_i} B_i$ is a bijection, and given that $[\theta_k]$ (respectively, $[\pi_k]$) is in $A_i$ if and only if $f[\theta_k]$ (respectively, $f[\pi_k]$) is in $B_i$, we must have that $A'_i\xrightarrow{f_i} B'_i$ is also a bijection.
 \end{proof}
 
 \begin{theorem}\label{theorem 8}
 For all $i \geq 0$, $\betti_{i,m}(S/M)=\betti_{i,m}(S/M')$.
 \end{theorem}
 
 \begin{proof}
 By Theorem \ref{theorem 6} and Corollary \ref{corollary 1.5}, we have that $\betti_{i,m}(S/M)=\#\{[\sigma]\in \mathbb{G}_u: \mdeg[\sigma]=m,\hdeg[\sigma]=i\}$. By the note following Theorem \ref{theorem exist} and Corollary \ref{corollary 1.5},
 \[\betti_{i,m}(S/M')=\#\{[\sigma]\in \mathbb{F}_u: \mdeg[\sigma]=m,\hdeg[\sigma]=i\}.\]
 Finally, by Theorem \ref{theorem 7}, we have that $\betti_{i,m}(S/M)=\betti_{i,m}(S/M')$.
 \end{proof}
 
 \section{Characterization Theorems}
 
 In this section, $S=k[x_1,\ldots,x_n]$ for an arbitrary, but fixed $n\geq 1$. Let $l=x_1^{\alpha_1}\ldots x_n^{\alpha_n}$, and $l'=x_1^{\beta_1}\ldots x_n^{\beta_n}$ be two monomials. We say that  $l$ \textbf{strongly divides} $l'$, 
 if $\alpha_i<\beta_i$, whenever $\alpha_i\neq 0$. 
 
 \begin{theorem}\label{char theorems 1}
 Let $M$ be minimally generated by $G$. Let $\mathbb{F}$ be a minimal resolution of $S/M$. Suppose that the basis of $\mathbb{F}$ contains an element $[\sigma]$, with $\hdeg[\sigma]=n$. Then there exists a subset $L$ of $G$, such that:
 
 \begin{enumerate}[(i)]
 \item $L$ is dominant.
 \item $\# L=n$.
 \item $\lcm(L)=\mdeg[\sigma]$.
 \item No element of $G$ strongly divides $\lcm(L)$.
 \end{enumerate} 
 \end{theorem}
 
 \begin{proof}
 Let $l=\mdeg[\sigma]$, and let $\{l_1,\ldots,l_q\}$ be the set of all monomials in $G$ dividing $l$. Let $M_l=(l_1,\ldots,l_q)$. By [GHP, Theorem 2.1], the minimal resolution $\mathbb{F}_l$ of $S/M_l$ is the subresolution of $\mathbb{F}$, defined by the basis elements of $\mathbb{F}$ whose multidegrees are divisors of  $l$. Thus, $[\sigma]$ is in the basis of $\mathbb{F}_l$ and, hence, $\pd(S/M_l)=n$.\\
 Suppose, by means of contradiction, that for some $1\leq i\leq n$, $x_i\nmid l$. Then none of the generators of $M_l$ is divisible by $x_i$, and $M_l$ is a monomial ideal in $k[x_1,\ldots,\widehat{x_i},\ldots,x_n]$. It follows from the Hilbert Syzygy theorem that $\pd(S/M_l) \leq n-1$, an absurd. Thus, $l$ must be of the form $l=x_1^{\alpha_1}\ldots x_n^{\alpha_n}$, where $\alpha_i \geq 1$, for all $i$. \\
 Let $M'_l=(l'_1,\ldots,l'_q)$ be the twin ideal of $M_l$. If we express $l_1,\ldots,l_q$ in the form: 
 
 \[\begin{array}{ccc}
 l_1 &= &x_1^{\alpha_{11}}\ldots x_n^{\alpha_{1n}},\\
 & \vdots& \\
 l_q &=&x_1^{\alpha_{q1}}\ldots x_n^{\alpha_{qn}},
 \end{array} \]
 then $l'_1,\ldots,l'_q$ must be of the form:
 
\[\begin{array}{ccc}
 l'_1&=&x_1^{\beta_{11}}\ldots x_n^{\beta_{1n}},\\
 & \vdots &  \\
 l'_q&=&x_1^{\beta_{q1}}\ldots x_n^{\beta_{qn}},
 \end{array}\]
 where, for all $1\leq i\leq q$, and all $1 \leq j\leq n$,
  $\beta_{ij}=
 \begin{cases} 
 \alpha_j \text{, if } \alpha_{ij}=\alpha_j\\
 0 \text{, if } \alpha_{ij} \neq \alpha_j.
 \end{cases}$
 
 In particular, $x_j$ appears with exponent either $\alpha_j$ or $0$ in the factorization of each generator $l'_1,\ldots,l'_q$. Let us make the change of variables $y_1=x_1^{\alpha_1},\ldots,y_n=x_n^{\alpha_n}$. Then 
 $\l'_1,\ldots,l'_q$ can be represented in the form:
 \[\begin{array}{ccc}
 l'_1&=&y_1^{\delta_{11}}\ldots y_n^{\delta_{1n}},\\
 & \vdots &  \\
 l'_q&=&y_1^{\delta_{q1}}\ldots y_n^{\delta_{qn}},
 \end{array}\]
 where, for all $1\leq i\leq q$, and all $1\leq j\leq n$,
  $\delta_{ij}=
 \begin{cases} 
 1 \text{, if } \alpha_{ij}=\alpha_j\\
 0 \text{, if } \alpha_{ij} \neq \alpha_j.
 \end{cases}$
 Hence, we can interpret $M'_l$ as a squarefree monomial ideal in $k[y_1,\ldots,y_n]$, such that $\pd\left(\dfrac{k[y_1,\ldots,y_n]}{M'_l}\right)=n$. By Theorem \ref{theorem 1.2}, $M'_l=(y_1,\ldots,y_n)$. Therefore, 
 $y_1,\ldots,y_n \in \{l'_1,\ldots,l'_q\}$. After reordering the subindices, we may assume that $l'_1=y_1=x_1^{\alpha_1},\ldots,l'_n=y_n=x_n^{\alpha_n}$. This means that:
 \[
 \begin{array}{cclll}
 l_1 &= &x_1^{\alpha_1}x_2^{\alpha_{12}} & \ldots  & x_n^{\alpha_{1n}} \text{, with } \alpha_{1i}<\alpha_i \text{, for all } i\neq 1.\\
 l_2 &= &x_1^{\alpha_{21}}x_2^{\alpha_2} &\ldots  & x_n^{\alpha_{2n}} \text{, with } \alpha_{2i}<\alpha_i \text{, for all } i\neq 2.\\
 
     &\vdots\\
 l_n &= &x_1^{\alpha_{n1}} \ldots                & x_{n-1}^{\alpha_{n n-1}} & x_n^{\alpha_{n}} \text{, with } \alpha_{ni}<\alpha_i \text{, for all } i\neq n.\\
    
 \end{array}\]
 This implies that each $x_i$ appears with exponent $\alpha_i$ in the factorization of $l_i$, and with exponent $\alpha_{ki}<\alpha_i$ in the factorization of $l_k$, if $k\neq i$. It follows that the set $L=\{l_1,\ldots,l_n\}$ is dominant (where $l_i$ is dominant in $x_i$), of cardinality $n$, which proves (i) and (ii). Moreover, $\lcm(L)=\lcm(l_1,\ldots,l_n)=x_1^{\alpha_1}\ldots x_n^{\alpha_n}=l$, which proves (iii). \\
 Finally, suppose that there is an element $m\in G$ that strongly divides $\lcm(L)=l$. Then, $M'_l$ contains $m'=1$ among its generators. Thus, $M'_l=(1)=S$, and $\pd(S/M'_l)\neq n$, a contradiction. Therefore, no element of $G$ strongly divides $\lcm(L)$, which proves (iv). 
 \end{proof}
 
 The next theorem is essentially a converse to Theorem \ref{char theorems 1}.
 
 \begin{theorem}\label{char theorems 2}
 Let $M$ be minimally generated by $G$. Suppose that there is a subset $L$ of $G$ with the following properties:
 
 \begin{enumerate}[(i)]
 \item $L$ is dominant.
 \item $\# L=n$.
 \item No element of $G$ strongly divides $\lcm(L)$.
 \end {enumerate}
 Then, there is a basis element $[\sigma]$ of the minimal resolution $\mathbb{F}$ of $S/M$, such that $\hdeg[\sigma]=n$, and $\mdeg[\sigma]=\lcm(L)$. Moreover, if $[\tau]$ is in the basis of $\mathbb{F}$, and $[\tau]\neq [\sigma]$,
 then $\mdeg[\tau]\neq \lcm(L)$.
 \end{theorem}
 
 \begin{proof}
 Let $L=\{l_1,\ldots,l_n\}$, where each $l_i$ is dominant in $x_i$ and let $\lcm(L)=l$. Let $G_l=\{m\in G:m\mid l\}$, and let $M_l$ be the ideal generated by $G_l$. If we express $l_1,\ldots,l_n$ in the form:
 
 \[\begin{array}{ccc}
 l_1 &= &x_1^{\alpha_{11}}\ldots x_n^{\alpha_{1n}},\\
 & \vdots& \\
 l_n &=&x_1^{\alpha_{n1}}\ldots x_n^{\alpha_{nn}},
 \end{array} \]
 then $l=\prod \limits_{i=1}^n x_i^{\alpha_{ii}}$.\\
 Let $M'_l$ be the twin ideal of $M_l$. Then $M'_l$ contains $l'_1=x_1^{\alpha_{11}},\ldots,l'_n=x_n^{\alpha_{nn}}$ among its generators. Moreover, it follows from (iii) that if $m \in G_l$, there must be an index $i$, such that $x_i$ appears with exponent $\alpha_{ii}$ in the factorization of $m$. Therefore, $m$ is divisible by $l'_i=x_i^{\alpha_{ii}}$. Thus $M'_l=(x_1^{\alpha_{11}},\ldots, x_n^{\alpha_{nn}})$. Hence, $\betti_{n,l}(S/M'_l)=1$, and 
 $\betti_{k,l}(S/M'_l)=0$, for $k<n$. By Theorem \ref{theorem 8}, $\betti_{n,l}(S/M_l)=1$, and $\betti_{k,l}(S/M_l)=0$, for $k<n$. By [GHP, Theorem 2.1], $\betti_{n,l}(S/M)=1$, and $\betti_{k,l}(S/M)=0$, for $k<n$. This implies that there is an element $[\sigma]$ in the basis of $\mathbb{F}$, with $\hdeg[\sigma]=n$, and $\mdeg[\sigma]= l=\lcm(L)$. Moreover, since $\sum\limits_{k=1}^n \betti_{k,l}(S/M_l)=1$, it follows that every basis element $[\tau]\neq [\sigma]$ must have multidegree $\mdeg[\tau]\neq l=\lcm(L)$.
 \end{proof}
 
 \begin{corollary}
 Let $M$ be minimally generated by $G$. The following statements are equivalent:
 
 \begin{enumerate}[(i)]
 \item $\pd(S/M)=n$.
 \item $G$ contains a dominant set $L$ of cardinality $n$, such that no element in $G$ strongly divides $\lcm(L)$.
 \end{enumerate}
 \end{corollary}
 
 \begin{proof}
 (i) $\Rightarrow$ (ii) Since $\pd(S/M)=n$, the minimal resolution of $S/M$ must contain a basis element in homological degree $n$. By Theorem \ref{char theorems 1}, there exists $L\subseteq G$, such that $L$ is dominant, $\# L=n$, and no monomial in $G$ strongly divides $\lcm(L)$.\\
 (ii)$\Rightarrow$(i) By Theorem \ref{char theorems 2}, the minimal resolution of $S/M$ must contain a basis element in homological degree $n$. By the Hilbert Syzygy theorem, $\pd(S/M)=n$.
 \end{proof}

 \section{Main Results}
 
 The following notation is fixed for the rest of this section. Let $n$ be the number of variables of $S$. Let $G$ be the minimal generating set of $M$, and let $\mathscr{D}_M$ denote the class
 $\mathscr{D}_M=
 \{D \subseteq G: D$ is a dominant set of cardinality $n$, such that no generator of $G$ strongly divides $\lcm(D)$$\}$.
 
 \begin{theorem}\label{main results 1}
 Suppose that $[\sigma]$ is a basis element of a minimal resolution of $S/M$, such that $\mdeg[\sigma]=l$, and $\hdeg[\sigma]=n$. Then 
 \[\betti_{i,l}(S/M)=\begin{cases}
 				1, \text{ if } i = n\\
				0, \text{ otherwise.}
				\end{cases}\]
 \end{theorem}
 
 \begin{proof}
 
 By Theorem \ref{char theorems 1}, there is a set $L\in \mathscr{D}_M$, such that $\lcm(L)=l$. Now, our statement follows from Theorem \ref{char theorems 2}.
\end{proof}

The next corollary gives the multigraded Betti numbers of $S/M$ in homological degree $n$.
 
 \begin{corollary}\label{main results 2}
 For an arbitrary monomial ideal $M$ of $S$, we have that
 
 \[\betti_{n,l}(S/M)=\begin{cases}
 				1 \text{, if there is } D\in \mathscr{D}_M \text{, with }l=\lcm(D) \\
				0 \text{, otherwise.}
				\end{cases}\]
 \end{corollary} 
 
 \begin{proof}
 Suppose that there exists $L\in \mathscr{D}_M$, such that $\lcm(L)=l$. By Theorems \ref{char theorems 2} and \ref{main results 1}, $\betti_{n,l}(S/M)=1$. Now, suppose that there is no $L \in \mathscr{D}_M$, with $\lcm(L)=l$. Let us assume, by means of contradiction, that $\betti_{n,l}(S/M)\geq 1$. Then, the minimal resolution of $S/M$ must have a basis element $[\sigma]$, with $\hdeg[\sigma]=n$, $\mdeg[\sigma]=l$. By Theorem \ref{char theorems 1}, there is a dominant set 
 $L \subseteq G$, with $\# L=n$, such that no monomial of $G$ strongly divides $\lcm(L)$, and with $\lcm(L)= \mdeg[\sigma]=l$, an absurd. Thus, $\betti_{n,l}(S/M)=0$.
  \end{proof}
 
Next, we give the total and graded Betti numbers of $S/M$, in homological degree $n$.
 
 \begin{corollary}\label{main results 3} 
 Let $L=\{l: l=\lcm(D) \text{, for some } D\in \mathscr{D}_M\}$. For every $j\geq 0$, let $L_j=\{l\in L:\deg l=j\}$. Then
 \begin{enumerate}[(i)]
 \item $\betti_n(S/M)=\# L$,
 \item $\betti_{n,j}(S/M)= \# L_j$.
  \end{enumerate}
 \end{corollary}
 
 \begin{proof}
 \begin{enumerate}[(i)]
 \item By Corollary \ref{main results 2}, $\betti_n (S/M)=\sum\limits_{l\in L}\betti_{n,l}(S/M)=\# L$.
 \item By Corollary \ref{main results 2}, $\betti_{n,j}(S/M)=\sum\limits_{l\in L_j}\betti_{n,l}(S/M)=\# L_j.$
\end{enumerate}
 \end{proof}
 
 \begin{example}\label{example 3} 
 Let $M$ be minimally generated by $G=\{x_1^6x_2,x_1^5x_2^3,x_2^4,x_1x_3^4\}$. Of the $4$ subsets of $G$ of cardinality $3$, exactly two of them are in $\mathscr{D}_M$, namely, $L_1=\{x_1^6x_2,x_1^5x_2^3,x_1x_3^4\}$, and $L_2=\{x_1^5x_2^3,x_2^4,x_1x_3^4\}$. Since $\lcm(L_1)=x_1^6x_2^3x_3^4$, and $\lcm(L_2)=x_1^5x_2^4x_3^4$, we have $\betti_{3,x_1^6x_2^3x_3^4}(S/M)=1=\betti_{3,x_1^5x_2^4x_3^4}(S/M)$. Moreover, since $\lcm(L_1)\neq \lcm(L_2)$, $\betti_3(S/M)=2$, and since $\deg L_1=\deg L_2=13$, $\betti_{3,13}(S/M)=2$.
 \end{example}
 
 In the next corollary, we give the Betti numbers of an arbitrary monomial ideal in three variables. Trivariate monomial resolutions have been completely described [Mi,MS], but our aim is to show how easily we can obtain the Betti numbers, once we know the elements of $\mathscr{D}_M$.
 
 \begin{corollary}\label{main results 4}
 Let $M$ be a monomial ideal in $3$ variables, minimally generated by $q$ monomials. Let $L=\{l:l=\lcm(D) \text{, for some } D\in \mathscr{D}_M\}$. Then, $\betti_0(S/M)=1$; $\betti_1(S/M)=q$; $\betti_2(S/M)=\# L+q-1$; $\betti_3(S/M)=\# L$.
 \end{corollary}
 
 \begin{proof}
 If $[\sigma]$ is a Taylor symbol of $\mathbb{T}_M$, with homological degree either 0 or 1, then the multidegree of $[\sigma]$ is not shared by any other Taylor symbol of $\mathbb{T}_M$. It follows that $[\sigma]$ is a basis element of any minimal resolution of $S/M$, obtained from $\mathbb{T}_M$ by means of consecutive cancellations. Thus, $\betti_0(S/M)=1$ and $\betti_1(S/M)=q$. That $\betti_3(S/M)=\# L$ is the content of Corollary \ref{main results 3} (i). Finally, since the Euler characteristic of $S/M$ equals 0, we have that $\betti_2(S/M)=\betti_3(S/M)+\betti_1(S/M)-\betti_0(S/M)$.
 \end{proof}
 
 \begin{example}
 $M=(x_1^6x_2,x_1^5x_2^3,x_2^4,x_1x_3^4)$. In Example \ref{example 3}, we showed that $\betti_3(S/M)=2$. Since $M$ is minimally generated by $4$ monomials, we must have that $\betti_0(S/M)=1$; $\betti_1(S/M)=4$; and $\betti_2(S/M)=\betti_3(S/M)+\betti_1(S/M)-\betti_0(S/M)=2+4-1=5$.
 \end{example}
 
 Our next goal is to show that monomial ideals with large projective dimension satisfy the inequality $\sum \limits_{i=0}^n \betti_i(S/M)\geq 2^n$. The next lemma will be a useful tool to prove this fact.
 
 \begin{lemma}\label{lemma 1}
 Let $M$ be minimally generated by $G$. Suppose that $G$ contains a dominant set $D$ of cardinality $n$, such that no element of $G$ strongly divides $\lcm(D)$. Let $1\leq i_1<\ldots<i_q\leq n$, where $1\leq q\leq n$. Let $\mathscr{A}$ be the class of all subsets $A$ of $G$ satisfying the following conditions:
 
 \begin{enumerate}[(i)]
 \item if $j \in \{i_1,\ldots,i_q\}$, $x_j$ appears with the same exponent in the factorizations of $\lcm(D)$ and $\lcm(A)$.
 \item if $j \notin \{i_1,\ldots,i_q\}$, $x_j$ appears with smaller exponent in the factorization of $\lcm(A)$ than in the factorization of $\lcm(D)$.
 \end{enumerate}
 Then, $\#\mathscr{A}$ is odd.
 \end{lemma}
 
 \begin{proof}
 Let $\lcm(D)=x_1^{\alpha_1}\ldots x_n^{\alpha_n}$. Since each of the $n$ monomials of $D$ must be dominant in one of the $n$ variables of $S$, $D$ can be represented in the form $D=\{m_1,\ldots,m_n\}$, where, for all 
 $1\leq i,j \leq n$, $m_i=x_1^{\beta_{i1}}\ldots x_n^{\beta_{in}}$, and 
 
\[ \begin{cases}
\beta_{ij}=\alpha_j \text{, if } i=j\\
\beta_{ij}< \alpha_j \text{, if } i\neq j.
 \end{cases}\]
 Let $D'=\{m_{i_1},\ldots,m_{i_q}\}$, and let 
 $B=\{l \in G: l\mid \lcm(D)$ and if  $x_j$ appears with the same exponent in the factorizations of $l$ and $\lcm(D)$, then $j\in \{i_1,\ldots,i_q\}\}$. Notice that $D' \subseteq B$. \\
 Let $L=B\setminus D'$. Then $B$ can be expressed as the disjoint union $B=L \cup D'$. \\
 For each $L' \subseteq L$, let $D_{L'}$ be the smallest subset of $D'$, such that $L' \cup D_{L'} \in \mathscr{A}$. Let 
 $\mathscr{A}_{L'}=\{L' \cup H: D_{L'} \subseteq H \subseteq  D'\}$. Notice that $\mathscr{A}_{L'} \subseteq \mathscr{A}$. On the other hand, if $A \in \mathscr{A}$, then $A \subseteq B=L \cup D'$. It follows that $A$ can be expressed as the disjoint union $A=(A \cap L ) \cup (A\cap D')$.
 If we define $L_A=A \cap L$, and $H_A= A \cap D'$, then $A \in \mathscr{A}_{L_A}$. Thus, $\mathscr{A} = \bigcup\limits_{L' \subseteq L} \mathscr{A}_{L'}$. Notice that if $L'$ and $L''$ are different subsets of $L$, then the families $\mathscr{A}_{L'}$ and $\mathscr{A}_{L''}$ are disjoint. Therefore, $\# \mathscr{A}= \sum\limits_{L' \subseteq L} \# \mathscr{A}_{L'}$. We will show that  $\sum\limits_{L' \subseteq L} \# \mathscr{A}_{L'}$ is odd.\\
 Let $L'$ be an arbitrary subset of $L$, and let $p= \# D_{L'}$. Then the number of sets $H$, such that $D_{L'} \subseteq H \subseteq D'$, is $\sum\limits_{i=0}^{q-p} {q-p \choose i}=2^{q-p}$. That is, $\# \mathscr{A}_{L'}=2^{q-p}$.
 If $L'= \varnothing$, then $D_{L'}=D'$, and $p=q$. Therefore, $\# \mathscr{A}_{\varnothing}= 2^0=1$. Now, suppose that $L' \neq \varnothing$. Let $l \in L'$. Then $l$ divides $\lcm(D)$, but not strongly. This means that $l$ factors as $l= x_1^{\gamma_1} \ldots x_n^{\gamma_n}$, where $\gamma_j \leq \alpha_j$ for all $j$, and $\gamma_k = \alpha_k$ for some $1 \leq k \leq n$. Moreover, since $l \in L' \subseteq L \subseteq B$, we have that 
 $k \in \{i_1,\ldots,i_q\}$; say $k=i_r$. It follows that $m_{i_r} \notin D_{L'}$, and thus, $p= \# D_{L'} < \# D'=q$ which, in turn, implies that $\# \mathscr{A}_{L'}=2^{q-p}$ is even. Finally, $\#\mathscr{A}$ is a finite sum, all of whose terms are even, with the only exception of 
 $\# \mathscr{A}_{\varnothing}= 1$. We conclude that $\#\mathscr{A}= \sum\limits_{L' \subseteq L} \# \mathscr{A}_{L'}$ is odd.
 \end{proof}
 
 \begin{theorem}\label{theorem 2}
 If $\pd(S/M)=n$, then $\sum \limits_{i=0}^n \betti_i(S/M)\geq 2^n$.
 \end{theorem}
 
 \begin{proof}
 By Theorem \ref{char theorems 1}, the minimal generating set $G$ of $M$ contains a dominant set $D$ of cardinality $n$, such that no element of $G$ strongly divides $\lcm(D)$. 
 Let $1 \leq i_1<\ldots<i_q \leq n$, where $1\leq q\leq n$. Let 
 $\mathscr{A}'$ be the class of all Taylor symbols $[A]$ of $\mathbb{T}_M$ satisfying the following conditions:
 
 \begin{enumerate}[(i)]
 \item If $j \in \{i_1,\ldots,i_q\}$, $x_j$ appears with the same exponent in the factorizations of $\mdeg[D]$ and $\mdeg[A]$.
 \item If $j \notin \{i_1,\ldots,i_q\}$, $x_j$ appears with smaller exponent in the factorization of $\mdeg[A]$ than in the factorization of $\mdeg[D]$.
 \end{enumerate}
 
 Notice that there is a bijective correspondence $f$ between the subsets of $G$ and the Taylor symbols of $\mathbb{T}_M$, given by $A \leftrightarrow [A]$, where $\lcm(A)=\mdeg[A]$. It follows that the restriction 
 $f\restriction_{\mathscr{A}}$ determines a bijective correspondence between the class $\mathscr{A}$ defined in Lemma \ref{lemma 1}, and the class $\mathscr{A}'$, just introduced. By Lemma \ref{lemma 1}, we have that 
 $\#\mathscr{A}'$ is odd.\\
 Let $\mathbb{F}$ be a minimal resolution of $S/M$, obtained from $\mathbb{T}_M$ by means of consecutive cancellations. If $0\rightarrow S[A]\rightarrow S[A']\rightarrow 0$ is one such cancellation, and either $[A]$ or $[A']$ is in 
 $\mathscr{A}'$, then the other Taylor symbol must also be in $\mathscr{A}'$, for $\mdeg[A]=\mdeg[A']$. In other words, each consecutive cancellation eliminates either $0$ or $2$ Taylor symbols from $\mathscr{A}'$. Hence, after making all the consecutive cancellations that lead to $\mathbb{F}$, we will have eliminated an even number of Taylor symbols from the family $\mathscr{A}'$. Since $\mathscr{A}'$ has odd cardinality, the basis of $\mathbb{F}$ must contain at least one element of $\mathscr{A}'$.\\
 Let $\mathscr{U}$ be the class of all strictly increasing sequences $1\leq i_1<\ldots<i_q\leq n$, where $0\leq q\leq n$. Let $\mathscr{V}$ be the basis of $\mathbb{F}$. We define the application $g: \mathscr{U}\rightarrow \mathscr{V}$ as follows:
 $g(\{i_1,\ldots,i_q\})=[A_{i_1,\ldots,i_q}]$, where $[A_{i_1,\ldots,i_q}]$ is a Taylor symbol of the basis of $\mathbb{F}$ satisfying (i) and (ii).\\
 Suppose that $\{k_1,\ldots,k_s\}$ and $\{j_1,\ldots,j_r\}$ are different sequences of $\mathscr{U}$. Then there is an integer that belongs to one of the sequences but not to the other, say 
 $k_t\in \{k_1,\ldots,k_s\}\setminus \{j_1,\ldots,j_r\}$. By (i), $x_{k_t}$ appears with the same exponent in the factorizations of $\mdeg[D]$ and $\mdeg[A_{k_1,\ldots,k_s}]$. By (ii), $x_{k_t}$ appears with smaller exponent in the factorization of $\mdeg[A_{j_1,\ldots,j_r}]$ than in the factorization of $\mdeg[D]$. This means that $\mdeg[A_{k_1,\ldots,k_s}]\neq \mdeg[A_{j_1,\ldots,j_r}]$. In particular, $[A_{k_1,\ldots,k_s}]\neq [A_{j_1,\ldots,j_r}]$, and $g$ is 
 one-to-one.\\
 Notice that the number of sequences in $\mathscr{U}$ is $\# \mathscr{U}=\sum\limits_{q=0}^n {n \choose q} =2^n$. Since $g$ is one-to-one, $\sum\limits_{i=0}^n \betti_i(S/M) = \# \mathscr{V} \geq \# \mathscr{U}=2^n.$
 \end{proof}

The final main result that we intend to prove states that Artinian monomial ideals $M$ (equivalently, ideals $M$ with $\codim(S/M)=n$) for which $\betti_n(S/M)=1$, are complete intersections. The proof of this fact requires the following lemma.
 
 \begin{lemma}\label{lemma 4.9}
 Let $M$ be minimally generated by $G=\{x_1^{\alpha_1},\ldots,x_n^{\alpha_n},l_1,\ldots,l_q\}$, where $q\geq 1$, $l_1,\ldots,l_q$ are divisible by $x_n$, and $\alpha_1,\ldots,\alpha_n\geq 1$. Then $\betti_n(S/M)\geq 2$.
 \end{lemma}
 
 \begin{proof}
 By induction on $n$. If $n=1$, there is nothing to prove.\\
 If $n=2$, $G$ must be of the form $G=\{x_1^{\alpha_1},x_2^{\alpha_2},x_1^{\beta_1}x_2^{\gamma_1},\ldots,x_1^{\beta_q}x_2^{\gamma_q}\}$, where $1\leq\beta_1<\beta_2<\ldots<\beta_q<\alpha_1$, and $1\leq\gamma_q<\gamma_{q-1}<\ldots<\gamma_1<\alpha_2$. Then $\{x_1^{\alpha_1},x_1^{\beta_q}x_2^{\gamma_q}\}$, $\{x_2^{\alpha_2},x_1^{\beta_1}x_2^{\gamma_1}\}$ are dominant sets of cardinality $2$ that are not strongly divisible by any element of $G$. Since $\lcm(x_1^{\alpha_1},x_1^{\beta_q}x_2^{\gamma_q})\neq \lcm(x_2^{\alpha_2},x_1^{\beta_1}x_2^{\gamma_1})$, it follows from Corollary \ref{main results 3} (i), that  $\betti_2\left(\dfrac{k[x_1,x_2]}{M}\right)\geq 2$.\\
 Suppose that the theorem holds for $n=k$. Let $M$ be an ideal of $k[x_1,\ldots,x_{k+1}]$ minimally generated by $G=\{x_1^{\alpha_1},\ldots,x_{k+1}^{\alpha_{k+1}},l_1,\ldots,l_q\}$, where $\alpha_1,\ldots,\alpha_{k+1}\geq 1$, 
 $q\geq 1$, and $l_1,\ldots,l_q$ are divisible by $x_{k+1}$.\\
 Let $L=\{l_i: 1 \leq i \leq q$, and $x_1 \mid l_i\}$. First, let us consider the case $L=\{l_1,\ldots,l_q\}$. Let $l\in L$ be such that exponent with which $x_1$ appears in the factorization of $l$ is less than or equal to the exponent with which $x_1$ appears in the factorization of any other monomial of $L$. Then $D_1=\{l,x_2^{\alpha_2},\ldots,x_{k+1}^{\alpha_{k+1}}\}$ is a dominant set of cardinality $k+1$ such that no monomial in $G$ strongly divides $\lcm(D_1)$. \\
 Likewise, let $l'\in L$ be such that the exponent with which $x_{k+1}$ appears in the factorization of $l'$ is less than or equal to the exponent with which $x_{k+1}$ appears in the factorization of any other monomial of $L$. Then 
 $D_2=\{x_1^{\alpha_1},\ldots,x_k^{\alpha_k},l'\}$ is a dominant set of cardinality $k+1$ such that no monomial in $G$ strongly divides $\lcm(D_2)$. \\
 Since $\lcm(D_1)\neq \lcm(D_2)$, it follows from Corollary \ref{main results 3} (i) that $\betti_{k+1}\left(\dfrac{k[x_1,\ldots,x_{k+1}]}{M}\right)\geq 2$.\\
 Now, let us consider the case $L\subsetneq \{l_1,\ldots,l_q\}$. Let $\{l_{i_1},\ldots,l_{i_r}\}$ be the set of all monomials in $\{l_1,\ldots,l_q\}$ that are not divisible by $x_1$. Let $M'$ be the ideal of $k[x_2,\ldots,x_{k+1}]$, minimally generated by $G'=\{x_2^{\alpha_2},\ldots,x_{k+1}^{\alpha_{k+1}},l_{i_1},\ldots,l_{i_r}\}$. Notice that $\alpha_2,\ldots,\alpha_{k+1}\geq 1$; $r\geq 1$, and $l_{i_1},\ldots,l_{i_r}$ are divisible by $x_{k+1}$. By induction hypothesis,
 $\betti_k\left(\dfrac{k[x_2,\ldots,x_{k+1}]}{M}\right)\geq 2$. By Corollary \ref{main results 3} (i), there are dominant subsets $E'_1, E'_2$ of $G'$ of cardinality $k$, with $\lcm(E'_1)\neq \lcm(E'_2)$, such that no monomial of $G'$ strongly divides $\lcm(E'_1)$ or $\lcm(E'_2)$.\\
 Let $L_1$ be the set of all monomials $m$ of $G$ that factor as $m=x_1^{\beta_1}\ldots x_{k+1}^{\beta_{k+1}}$, where $\beta_1 \geq 1$, and $x_2^{\beta_2}\ldots x_{k+1}^{\beta_{k+1}}$ strongly divides $\lcm(E'_1)$. (Note that $x_1^{\alpha_1}\in L_1$.) Let $\gamma_1$ be the smallest exponent with which $x_1$ appears in the factorization of any monomial of $L_1$, and let $m_1\in L_1$ be such that $x_1$ appears with exponent $\gamma_1$ in the factorization of $m_1$.
 Then $E_1=\{m_1\}\cup E'_1$ is a dominant set of cardinality $k+1$. Since no monomial of $G'$ strongly divides $\lcm(E'_1)$, it follows that no monomial of $G$ strongly divides $\lcm(E_1)$.
 Likewise, let $L_2$ be the set of all monomials $m$ of $G$ that factor as $m=x_1^{\beta_1}\ldots x_{k+1}^{\beta_{k+1}}$, where $\beta_1 \geq 1$, and $x_2^{\beta_2}\ldots x_{k+1}^{\beta_{k+1}}$ strongly divides $\lcm(E'_2)$. (Note that $x_1^{\alpha_1}\in L_2$.) Let $\gamma_2$ be the smallest exponent with which $x_1$ appears in the factorization of any monomial of $L_2$, and let $m_2\in L_2$ be such that $x_1$ appears with exponent $\gamma_2$ in the factorization of $m_2$.
 Then $E_2=\{m_2\}\cup E'_2$ is a dominant set of cardinality $k+1$. Since no monomial of $G'$ strongly divides $\lcm(E'_2)$, it follows that no monomial of $G$ strongly divides $\lcm(E_2)$. The fact that $\lcm(E'_1)\neq \lcm(E'_2)$ implies that $\lcm(E_1) \neq \lcm(E_2)$ and, by Corollary \ref{main results 3} (i), $\betti_{k+1}\left(\dfrac{k[x_1,\ldots,x_{k+1}]}{M}\right)\geq 2$.
 \end{proof}
 
 \begin{theorem}\label{theorem 6.10}
Let $\codim(S/M)=n$. If $\betti_n(S/M)=1$, then $M=(x_1^{\alpha_1},\ldots,x_n^{\alpha_n})$, for some ${\alpha_1},\ldots,{\alpha_n}\geq 1$.
 \end{theorem}
 
 \begin{proof}
 We will prove the logically equivalent statement: if $M\neq (x_1^{\alpha_1},\ldots,x_n^{\alpha_n})$, then $\betti_n(S/M) \geq 2$. (Note that $\betti_n(S/M) \neq 0$, for $\codim(S/M)=n$.)\\
 The proof is by induction on $n$. If $n=1$, the result holds trivially.\\
 Let us assume now that the theorem holds for $n=k$. \\
 Suppose that $M\neq (x_1^{\alpha_1},\ldots,x_{k+1}^{\alpha_{k+1}})$ is an ideal in $S=k[x_1,\ldots,x_{k+1}]$, with $\codim(S/M)=k+1$. Let $G$ be the minimal generating set of $M$, and let $G'$ be the set of all monomials in $G$ that are divisible by $x_{k+1}$. Since $\codim(S/M)=k+1$, $G$ must contain monomials of the form $x_1^{\alpha_1},\ldots,x_{k+1}^{\alpha_{k+1}}$ among its generators. In particular, $G'\neq \varnothing$ because $x_{k+1}^{\alpha_{k+1}}\in G'$. Let $M'$ be the ideal in $k[x_1,\ldots,x_k]$, minimally generated by $G\setminus G'$. Since $x_1^{\alpha_1},\ldots ,x_k^{\alpha_k} \in G\setminus G'$, $\codim\left(\dfrac{k[x_1,\ldots,x_k]}{M'}\right)=k$.
  If $M'=(x_1^{\alpha_1},\ldots,x_k^{\alpha_k})$, then $M$ satisfies the hypotheses of  Lemma \ref{lemma 4.9}, and thus $\betti_{k+1}(S/M)\geq 2$.\\
 Suppose now that $M'\neq (x_1^{\alpha_1},\ldots,x_k^{\alpha_k})$. By induction hypothesis, $\betti_k\left(\dfrac{k[x_1,\ldots,x_k]}{M'}\right)\geq 2$. By Corollary \ref{main results 3} (i), there are two dominant subsets $D_1$, $D_2$ of $G\setminus G'$, of cardinality $k$, with $\lcm(D_1)\neq \lcm(D_2)$, such that no monomial in $G\setminus G'$ strongly divides $\lcm(D_1)$ or $\lcm(D_2)$.\\
 Let $G'_1$ be the set of all monomials $m$ of $G'$ that factor as $m=x_1^{\beta_1}\ldots x_{k+1}^{\beta_{k+1}}$, where $x_1^{\beta_1}\ldots x_k^{\beta_k}$ strongly divides $\lcm(D_1)$. (Note that $G'_1\neq \varnothing$, for 
 $x_{k+1}^{\alpha_{k+1}}\in G'_1$.) Let $l_1\in G'_1$ be such that the exponent with which $x_{k+1}$ appears in the factorization of $l_1$ is less than or equal to the exponent with which $x_{k+1}$ appears in the factorization of any other monomial of $G'_1$. Then $D_1\cup \{l_1\}$ is a dominant set of cardinality $k+1$, that is not strongly divisible by any monomial of $G$.
 Likewise, let $G'_2$ be the set of all monomials $m$ of $G'$ that factor as $m=x_1^{\beta_1}\ldots x_{k+1}^{\beta_{k+1}}$, where $x_1^{\beta_1}\ldots x_k^{\beta_k}$ strongly divides $\lcm(D_2)$. 
 (Note that $G'_2\neq \varnothing$, for 
 $x_{k+1}^{\alpha_{k+1}}\in G'_2$.) Let $l_2\in G'_2$ be such that the exponent with which $x_{k+1}$ appears in the factorization of $l_2$ is less than or equal to the exponent with which $x_{k+1}$ appears in the factorization of any other monomial of $G'_2$. Then $D_2\cup \{l_2\}$ is a dominant set of cardinality $k+1$, that is not strongly divisible by any monomial of $G$.
  Since $\lcm(D_1)\neq \lcm(D_2)$, we must have that $\lcm(D_1\cup\{l_1\})\neq \lcm(D_2\cup\{l_2\})$ and, by Corollary \ref{main results 3} (i), $\betti_{k+1}(S/M) \geq 2$.
 \end{proof}
 
 \begin{corollary}\label{equivalent}
 Suppose that $\betti_k(S/M)=1$, for some $1\leq k\leq n$. The following statements are equivalent:
 \begin{enumerate}[(i)]
 \item $S/M$ is Gorenstein.
 \item $S/M$ is Cohen-Macaulay.
 \item $\codim(S/M)=k$.\\
 In particular, when $k=n$, the conditions above are equivalent to
 \item $M$ is of the form $M=(x_1^{\alpha_1},\ldots,x_n^{\alpha_n})$, for some ${\alpha_1},\ldots,{\alpha_n}\geq 1$.
 \end{enumerate}
 \end{corollary}
 
 \begin{proof}
The equivalence between (i) and (ii) is proven in [Pe, Theorem 25.7]. 
Let 
 \[\mathbb{F}: 0\rightarrow F_n\xrightarrow{f_n} F_{n-1}\cdots F_1\xrightarrow{f_1} F_0 \xrightarrow{f_0} S/M \rightarrow 0\]
 be a minimal resolution of $S/M$. (If $\pd(S/M)<n$, some of the $F_i$ will be trivial.) Since $\betti_k(S/M)=1$, the basis of $F_k$ is of the form $\{ [\sigma] \}$. Suppose that $f_k [ \sigma ] =0$. 
 Then $\ker f_{k-1} = f_k(F_k)=0$, and hence, $\pd(S/M) \leq k-1$, an absurd. Thus, we must have that $f_k [ \sigma ] \neq0$, which implies that $\ker f_k=0$. Hence, $\pd(S/M) =k$.
 From this fact, the equivalence between (ii) and (iii) is immediate.
 
 Now suppose that $k=n$.
 
 (iii) $\Rightarrow$ (iv) This part follows from Theorem \ref{theorem 6.10}.  
 
 (iv) $\Rightarrow$ (iii) This follows from the fact that $M$ is a complete intersection minimally generated by $n$ monomials.
 \end{proof}

 \textit{Note}: It is proven in [Pe,Theorem 25.6] that if $\betti_k(S/M)=1$ and $S/M$ is Gorenstein, then the Betti numbers of $S/M$ are symmetric; that is, $\betti_i(S/M)=\betti_{k-i}(S/M)$. Corollary \ref{equivalent} sheds some 
 light on the case $k=n$. Specifically, it shows that the symmetry is due to the fact that $\betti_i(S/M)=\betti_i\left(S/(x_1^{\alpha_1},\ldots,x_n^{\alpha_n})\right)={n\choose i}={n\choose {n-i}}=\betti_{n-i}\left(S/(x_1^{\alpha_1},\ldots,x_n^{\alpha_n})\right)=\betti_{n-i}(S/M)$.

 \bigskip

\noindent \textbf{Acknowledgements}: Many thanks to my dear wife Danisa for getting up very early in the morning, day after day, to type this work. I cannot express with words how much I admire her.

\end{document}